\documentclass[12pt]{article}
\usepackage{a4wide}

\usepackage{amssymb,amsthm,amsmath}
\usepackage{hyperref}
\usepackage{verbatim}
\usepackage{graphicx}
\usepackage{tikz}
\usetikzlibrary{patterns}
\usetikzlibrary{decorations.pathreplacing}
\usepackage[hypcap]{caption}
\usepackage{color}
\definecolor{nicegreen}{rgb}{0,0.5,0}
\definecolor{blue3}{rgb}{.1,.0,.4}
\definecolor{red3}{rgb}{.9,.0,.1}
\hypersetup{colorlinks=true, linkcolor=red3, urlcolor=blue3, citecolor=blue, pdfpagemode=UseNone, pdfstartview=}

\DeclareMathOperator{\conv}{conv}

\if10     
\usepackage[mathlines]{lineno}
\newcommand*\patchAmsMathEnvironmentForLineno[1]{%
  \expandafter\let\csname old#1\expandafter\endcsname\csname #1\endcsname
  \expandafter\let\csname oldend#1\expandafter\endcsname\csname end#1\endcsname
  \renewenvironment{#1}%
     {\linenomath\csname old#1\endcsname}%
     {\csname oldend#1\endcsname\endlinenomath}}%
\newcommand*\patchBothAmsMathEnvironmentsForLineno[1]{%
  \patchAmsMathEnvironmentForLineno{#1}%
  \patchAmsMathEnvironmentForLineno{#1*}}%
\AtBeginDocument{%
\patchBothAmsMathEnvironmentsForLineno{equation}%
\patchBothAmsMathEnvironmentsForLineno{align}%
\patchBothAmsMathEnvironmentsForLineno{flalign}%
\patchBothAmsMathEnvironmentsForLineno{alignat}%
\patchBothAmsMathEnvironmentsForLineno{gather}%
\patchBothAmsMathEnvironmentsForLineno{multline}%
}
\linenumbers
\fi

\definecolor{modra}{rgb}{0,0,.8}
\definecolor{seda}{rgb}{.7,.7,.7}

\makeatletter\newtheorem*{rep@theorem}{\rep@title}\newcommand{\newreptheorem}[2]{\newenvironment{rep#1}[1]{\def\rep@title{#2~\ref{##1}}\begin{rep@theorem}}{\end{rep@theorem}}}\makeatother

\newtheorem{theorem}{Theorem}
\newreptheorem{theorem}{Theorem}

\newtheorem{lemma}[theorem]{Lemma}
\newtheorem{conjecture}[theorem]{Conjecture}

\theoremstyle{definition}

\newtheorem{observation}[theorem]{Observation}

\newcommand{\R}{\mathbb R}

\begin{document}
\title{Near equipartitions of colored point sets}

\author{
Andreas F. Holmsen\thanks{Department of Mathematical Sciences, KAIST, Daejeon, South Korea
     Email: {\tt  andreash@kaist.edu}
     Research partially supported by Swiss National Science Foundation
 Grants 200020-165977 and 200021-162884.
     }
\and
{Jan Kyn\v{c}l\thanks{Department of Applied Mathematics and Institute for Theoretical Computer
Science, Charles University, Prague. Email: {\tt kyncl@kam.mff.cuni.cz}. Supported by the grant no. 14-14179S of the Czech Science Foundation (GA\v{C}R)}}
\and
{Claudiu Valculescu\thanks{ EPFL, Lausanne, Switzerland. Email: {\tt adrian.valculescu@epfl.ch}. Research partially supported by Swiss National Science Foundation
 Grants 200020-165977 and 200021-162884.}}
}

\date{}
\maketitle

\begin{abstract}
Suppose that $nk$ points in general position in the plane are colored red and blue, with at least $n$ points of each color. We show that then there exist $n$ pairwise disjoint convex sets, each of them containing $k$ of the points, and each of them containing points of both colors.

We also show that if $P$ is a set of $n(d+1)$ points in general position in $\mathbb{R}^d$ colored by $d$ colors with at least $n$ points of each color, then there exist $n$ pairwise disjoint $d$-dimensional simplices with vertices in $P$, each of them containing a point of every color.

These results can be viewed as a step towards a common generalization of several previously known geometric partitioning results regarding colored point sets.
\end{abstract}

\section{Introduction}

In this note, we prove two results concerning partitions of colored point sets. We conjecture a common generalization of these results, as well as various other related results and conjectures~\cite{aichholzer, akyiama, kano}. First we establish some basic terminology.

\paragraph{Definitions.} 
We say that a finite set in $\mathbb{R}^d$ is in \emph{general position} if each of its subsets of size at most $d+1$ is affinely independent.  A partition of a finite set $X$ into $m$ parts is called an \emph{$m$-coloring} of $X$, the parts are called \emph{color classes}, and we also say that $X$ is \emph{$m$-colored}. We allow the color classes to be empty. A subset $Y\subseteq X$ is called \emph{$j$-colorful} if $Y$ contains points from at least $j$ distinct color classes. Let $X$ be a subset of $\mathbb{R}^d$, and $Y\subseteq X$. The convex hull of $Y$, denoted by $\conv Y$, is called an \emph{island} (\emph{spanned by $X$}) if $X \cap \conv Y = Y$. Equivalently, we say that the set $X$ \emph{spans} $Y$. If $\conv Y$ is an island spanned by $X$ and $|Y| = k$, then we also say that $\conv Y$ is a \emph{$k$-island}. If $Y\subseteq X$, we say that the island $\conv Y$ is \emph{$j$-colorful} if $Y$ is $j$-colorful. See Figure~\ref{fig:island}. Notice that when $X$ is in general position and $k\le d+1$, then a $k$-island spanned by $X$ is a $(k-1)$-dimensional simplex with vertices in $X$.

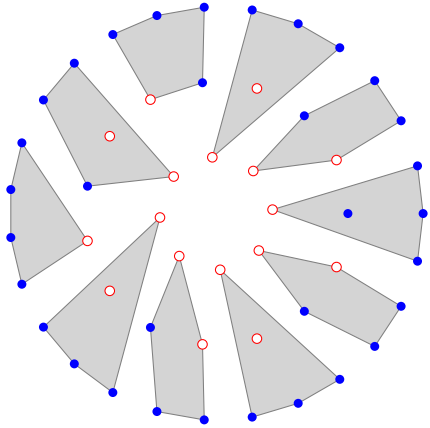
\begin{figure}
\centering
\begin{tikzpicture}
 \newcommand\Square[1]{+(-#1,-#1) rectangle +(#1,#1)}
  \begin{scope}
    \fill[gray, opacity=.34] (1*360/8 +4: .75cm) -- (1*360/15 : 1.75cm) -- (2*360/27 : 2.75cm) -- (3*360/27 : 2.75cm) -- (2*360/15 : 1.75cm) -- cycle;
    \draw[gray] (1*360/8 +4: .75cm) -- (1*360/15 : 1.75cm) -- (2*360/27 : 2.75cm) -- (3*360/27 : 2.75cm) -- (2*360/15 : 1.75cm) -- cycle;
   \end{scope}
 \begin{scope}
    \fill[gray, opacity=.34] (2*360/8 +4: .75cm) -- (4*360/27 : 2.75cm) -- (5*360/27 : 2.75cm) -- (6*360/27 : 2.75cm) -- cycle;
    \draw[gray] (2*360/8 +4: .75cm) -- (4*360/27 : 2.75cm) -- (5*360/27 : 2.75cm) -- (6*360/27 : 2.75cm) -- cycle;
   \end{scope}
  \begin{scope}
    \fill[gray, opacity=.34] (4*360/15 : 1.75cm) -- (7*360/27 : 2.75cm) -- (8*360/27 : 2.75cm) -- (9*360/27 : 2.75cm) -- (5*360/15 : 1.75cm) -- cycle;
    \draw[gray] (4*360/15 : 1.75cm) -- (7*360/27 : 2.75cm) -- (8*360/27 : 2.75cm) -- (9*360/27 : 2.75cm) -- (5*360/15 : 1.75cm) -- cycle;
   \end{scope}
 \begin{scope}
\fill[gray, opacity=.34] (3*360/8 +4: .75cm) -- (10*360/27 : 2.75cm) -- (11*360/27 : 2.75cm) -- (7*360/15 : 1.75cm) -- cycle;
\draw[gray] (3*360/8 +4: .75cm) -- (10*360/27 : 2.75cm) -- (11*360/27 : 2.75cm) -- (7*360/15 : 1.75cm) -- cycle;
   \end{scope}
 \begin{scope}
    \fill[gray, opacity=.34] (8*360/15: 1.75cm) -- (12*360/27 : 2.75cm) -- (13*360/27 : 2.75cm) -- (14*360/27 : 2.75cm) -- (15*360/27 : 2.75cm) -- cycle;
    \draw[gray] (8*360/15: 1.75cm) -- (12*360/27 : 2.75cm) -- (13*360/27 : 2.75cm) -- (14*360/27 : 2.75cm) -- (15*360/27 : 2.75cm) -- cycle;
   \end{scope}
 \begin{scope}
    \fill[gray, opacity=.34] (4*360/8+4 : .75cm) -- (16*360/27 : 2.75cm) -- (17*360/27 : 2.75cm) -- (18*360/27 : 2.75cm) -- cycle;
    \draw[gray] (4*360/8+4 : .75cm) -- (16*360/27 : 2.75cm) -- (17*360/27 : 2.75cm) -- (18*360/27 : 2.75cm) -- cycle;
   \end{scope}
 \begin{scope}
    \fill[gray, opacity=.34] (5*360/8+4 : .75cm) -- (10*360/15 : 1.75cm) -- (19*360/27 : 2.75cm) -- (20*360/27 : 2.75cm) -- (11*360/15 : 1.75cm) -- cycle;
    \draw[gray] (5*360/8+4 : .75cm) -- (10*360/15 : 1.75cm) -- (19*360/27 : 2.75cm) -- (20*360/27 : 2.75cm) -- (11*360/15 : 1.75cm) -- cycle;
   \end{scope}
 \begin{scope}
    \fill[gray, opacity=.34] (6*360/8+4 : .75cm) -- (21*360/27 : 2.75cm) -- (22*360/27 : 2.75cm) -- (23*360/27 : 2.75cm) -- cycle;
    \draw[gray] (6*360/8+4 : .75cm) -- (21*360/27 : 2.75cm) -- (22*360/27 : 2.75cm) -- (23*360/27 : 2.75cm) -- cycle;
   \end{scope}
 \begin{scope}
    \fill[gray, opacity=.34] (7*360/8+4 : .75cm) -- (13*360/15 : 1.75cm) -- (24*360/27 : 2.75cm) -- (25*360/27 : 2.75cm) -- (14*360/15 : 1.75cm) -- cycle;
    \draw[gray] (7*360/8+4 : .75cm) -- (13*360/15 : 1.75cm) -- (24*360/27 : 2.75cm) -- (25*360/27 : 2.75cm) -- (14*360/15 : 1.75cm) -- cycle;
   \end{scope}
 \begin{scope}
    \fill[gray, opacity=.34] (8*360/8+4 : .75cm) -- (26*360/27 : 2.75cm) -- (27*360/27 : 2.75cm) -- (1*360/27 : 2.75cm) -- cycle;
    \draw[gray] (8*360/8+4 : .75cm) -- (26*360/27 : 2.75cm) -- (27*360/27 : 2.75cm) -- (1*360/27 : 2.75cm) -- cycle;
   \end{scope}
  \begin{scope}[rotate = 4]
    \foreach \i in {1,...,8} { \draw[red,fill=white] (360*\i/8  :.75cm)  circle (1.8pt); }
  \end{scope}
  \begin{scope}[rotate = 0]
    \foreach \i in {1,...,27} { \fill[blue] (360*\i/27  :2.75cm) circle (1.7pt); }
  \end{scope}
  \begin{scope}[rotate = 0]
    \foreach \i in {2,4, 7, 10, 13, 15 } { \fill[blue] (360*\i/15  :1.75cm) circle (1.7pt); }
   \foreach \i in {1,3,5,6,8,9,11,12,14} { \draw[red,fill=white] (360*\i/15  :1.75cm)  circle (1.8pt); }
 \end{scope}
 \end{tikzpicture}
\caption{A $2$-colored set of 50 points spanning 10 pairwise disjoint $2$-colorful $5$-islands. \label{fig:island}}
\end{figure}

\paragraph{The results.}
Our first result concerns partitions of $2$-colored planar point sets into $2$-colorful subsets of $k$ points with disjoint convex hulls.

\begin{theorem}
\label{theorem_plane}
Let $k\ge 2$ and $n\ge 1$ be integers, and let $X$ be a $2$-colored point set in general position in $\mathbb{R}^2$. Suppose that $|X| = kn$ and that there are at least $n$ points in each color class. Then $X$ spans $n$ pairwise disjoint $2$-colorful $k$-islands.
\end{theorem}

Our second result concerns partitions of $d$-colored point sets in $\mathbb{R}^d$ into $d$-colorful subsets of $d+1$ points with disjoint convex hulls.

\begin{theorem}
\label{theorem_Rd}
Let $d\ge 2$ and $n\ge 1$ be integers, and let $X$ be a $d$-colored point set in general position in $\mathbb{R}^d$. Suppose that $|X| = (d+1)n$ and that there are at least $n$ points in each color class. Then $X$ spans $n$ pairwise disjoint $d$-colorful $(d+1)$-islands.
\end{theorem}

Both theorems can be seen as particular cases of the following common generalization:

\begin{conjecture}
\label{conj_main}
Let $d,k,m$ be integers satisfying $k,m\ge d\ge 2$. Let $X$ be an $m$-colored set of $kn$ points in general position in $\mathbb{R}^d$. Suppose that $X$ admits a partition into $n$ pairwise disjoint $d$-colorful $k$-tuples. Then $X$ spans $n$ pairwise disjoint $d$-colorful $k$-islands.
\end{conjecture}

The condition that $X$ admits a partition into $n$ pairwise disjoint $d$-colorful $k$-tuples can be stated equivalently as the following Hall-type condition on the sizes of the color classes.

\begin{lemma}
\label{lemma_hall}
Let $d,k,m$ be integers satisfying $k,m\ge d\ge 2$. Let $X$ be a set with $kn$ elements and let $X=X_1 \cup X_2 \cup \dots \cup X_m$ be an $m$-coloring of $X$. The set $X$ admits a partition into $n$ pairwise disjoint $d$-colorful $k$-tuples if and only if for every $t\in [d-1]$ and every subset $I\subset [m]$ with $|I|=t$ we have
\begin{equation}\label{eq_hall}
\sum_{i\in I} |X_i|\le (k-d+t)n.
\end{equation}
\end{lemma}

Lemma~\ref{lemma_hall} is purely combinatorial, and not geometric in nature. We provide its proof in Section~\ref{section_aux}, where in Lemma~\ref{lemma_merge} we also show how to reduce the problem in Conjecture~\ref{conj_main} to the case of $2d-1$ or $2d-2$ colors by merging some color classes.

Note that by Lemma~\ref{lemma_hall}, the conditions on the sizes of the color classes stated in Theorem~\ref{theorem_plane} and Theorem~\ref{theorem_Rd} are necessary for the existence of a partition into $n$ pairwise disjoint $d$-colorful $k$-tuples.

Theorem~\ref{theorem_plane} confirms Conjecture~\ref{conj_main} for $k\ge m=d=2$. This, together with Lemma~\ref{lemma_merge}, implies Conjecture~\ref{conj_main} for $d=2$ and arbitrary $k,m\ge 2$.
The proof of Theorem~\ref{theorem_plane} is elementary and is based on a result by Kaneko, Kano, and Suzuki~\cite{kanekoKanoSuzuki}. The proof is given in Section~\ref{section_planar}.

Theorem~\ref{theorem_Rd} confirms Conjecture~\ref{conj_main} for $k-1=m=d\ge 2$. The proof of Theorem~\ref{theorem_Rd} is based on the continuous ham sandwich theorem and a special discretization argument~\cite{EH11_stronger_conclusion,kano} and it is given in Section~\ref{section_simplices}.

\paragraph{Relation to previous results.}

The classical ham sandwich theorem states that for any $d$ measures in $\mathbb{R}^d$ there is a hyperplane bisecting each of these $d$ measures simultaneously. The theorem is often used in two versions: a continuous version with ``nice''  measures (see Theorem~\ref{theorem_ham_sandwich}) and a discrete version with discrete measures or point sets. The discrete ham sandwich theorem has been a source of influence for further developments related to a wide range of geometric partitioning results for discrete point configurations. The following result is a typical example.

\begin{theorem}[Akiyama--Alon~\cite{akyiama}]\label{theorem_alon}
Let $d\ge 2$ and $n\ge 1$ be integers, and let $X$ be a $d$-colored point set in general position in $\mathbb{R}^d$. Suppose that $|X| = dn$ and that there are exactly $n$ points in each color class. Then $X$ spans $n$ pairwise disjoint $d$-colorful $d$-islands. 
\end{theorem}

The planar case of Theorem~\ref{theorem_alon} has the following generalization, conjectured by Kaneko and Kano~\cite{kaneko}, and proven independently by Bespamyatnikh et al.~\cite{besp}, Ito et al.~\cite{IUY00_2d} and Sakai~\cite{sakai}.

\begin{theorem}[Bespamyatnikh et al.~\cite{besp}, Ito et al.~\cite{IUY00_2d}, Sakai~\cite{sakai}]\label{theorem_sakai}
Let $A$ and $B$ be disjoint finite sets in $\mathbb{R}^2$ such that $A\cup B$ is in general position, $|A| = an$, and $|B|= bn$. Then there exist $n$ pairwise disjoint convex sets $C_1, C_2, \dots, C_n$ such that $|C_i\cap A| = a$ and $|C_i\cap B| = b$ for every $i\in[n]$.
\end{theorem}

Theorem~\ref{theorem_sakai} solves the case of Theorem~\ref{theorem_plane} when the size of each color class is divisible by $n$. There is also a continuous version of Theorem~\ref{theorem_sakai} due to Sakai~\cite{sakai}, which was generalized to arbitrary dimension by Sober{\'o}n~\cite{soberon}. Sober{\'o}n's proof of the continuous version relies on an ingenious application of power diagrams and Dold's theorem. Even further generalizations were obtained by Karasev et al.~\cite{spicy} and independently by Blagojevi{\'c} and Ziegler~\cite{blagojevic}. However, going from the continuous version to the discrete version seems to require, in many cases, a non-trivial approximation argument, and we do not see how the continuous results~\cite{blagojevic,spicy,soberon} could be used to settle our Conjecture~\ref{conj_main} for the case $m=d$.

In the discrete setting, it is natural to try to relax the divisibility condition on the sets $|A|$ and $|B|$ in Theorem~\ref{theorem_sakai}, and some partial results were obtained in~\cite{KK05_semibalanced,kanekoKanoSuzuki,KU07_general,KU10_balanced}.

Another recent example in this direction is the following generalization of Theorem~\ref{theorem_alon}, due to Kano and Kyn{\v c}l~\cite{kano}.

\begin{theorem}[Kano--Kyn{\v c}l,~\cite{kano}]\label{theorem_kyncl}
Let $d\ge 2$ and $n\ge 1$ be integers, and let $X$ be a $(d+1)$-colored point set in general position in $\mathbb{R}^d$. Suppose that $|X| = dn$ and that there are at most $n$ points in each color class. Then $X$ spans $n$ pairwise disjoint $d$-colorful $d$-islands. 
\end{theorem}


Note that by Lemma~\ref{lemma_hall}, the conditions on the sizes of the color classes stated in Theorem~\ref{theorem_alon} and Theorem~\ref{theorem_kyncl} are necessary for the existence of a partition into $n$ pairwise disjoint $d$-colorful $d$-tuples.

Theorem~\ref{theorem_alon} proves the case $k=m=d$ of Conjecture~\ref{conj_main}, while the case $k=m-1=d$ is answered by Theorem~\ref{theorem_kyncl}. 
The case $k=d$ and $m\ge d$ was originally conjectured by Kano and Suzuki~\cite[Conjecture 3]{kano}. The case $m\ge k=d=2$ was proved by Aichholzer et al.~\cite
{aichholzer} and by Kano, Suzuki and Uno~\cite{kanoSuzukiUno}.

In this note we are mostly concerned with the case $k\ge m=d$. For $d=2$, Theorem~\ref{theorem_sakai} covers the subcase where the cardinality of each $X_i$ is divisible by $n$. Kaneko, Kano, and Suzuki~\cite{kanekoKanoSuzuki} solved the subcase with $d=2$, $k$ odd and $\big\lvert |X_1| - |X_2| \big\rvert \le n$.

When $n$ is a power of $2$, the case $m=d$ of Conjecture~\ref{conj_main} can be obtained relatively easily by induction from the discrete ham sandwich theorem~\cite[Theorem 1.4.3]{M02_lectures}, proceeding like in Case 1 in the proof of Theorem~\ref{theorem_special_discrete_sandwich} in Section~\ref{section_simplices}. Thus the main contribution of this paper and the main difficulty in it consists in removing the divisibility assumptions for $n$ and the sizes of the color classes.

\section{Auxiliary results}
\label{section_aux}

\begin{proof}[{Proof of Lemma~\ref{lemma_hall}}]
Without loss of generality, we assume that $|X_1|\ge|X_2|\ge \dots\ge|X_m|$. Condition~\eqref{eq_hall} can then be stated using only $d-1$ inequalities as follows:
\begin{equation}\label{eq_hall_linear}
\forall t\in [d-1] \quad \sum_{i=1}^t |X_i|\le (k-d+t)n.
\end{equation}

The necessity of condition~\eqref{eq_hall_linear} follows from the fact that for every $t\in [d]$, every $d$-colorful $k$-tuple has at most $k-d+t$ elements in $X_1 \cup X_2 \cup \dots \cup X_t$, since it has at least $d-t$ elements in $X_{t+1} \cup X_{t+2}\cup \dots \cup X_m$.

We now prove the sufficiency of~\eqref{eq_hall_linear}. If $|X_1|=n$, then let $t=0$. Otherwise let $t\in[m]$ be the largest index such that $|X_t|>n$. For each $i\in[d]$, let $Y_i$ be an arbitrary $n$-element subset of $X_i$ if $i\le t$, otherwise let $Y_i=X_i$. Let $Y=\bigcup_{i=1}^m Y_i$; see Figure~\ref{fig:skyscraper}. We claim that $|Y|\ge dn$.
Indeed, by condition~\eqref{eq_hall_linear} and by the assumptions that $|X|=kn$ and $k\ge d$ (which may be regarded as condition~\eqref{eq_hall_linear} for $t=0$), we have
\[
|Y|=|X|-|X\setminus Y|=
kn-\sum_{i=1}^t (|X_i|-n) = kn+tn - \sum_{i=1}^t |X_i|\ge dn.
\]

We construct a partition of $X$ into $n$ $d$-colorful $k$-tuples as follows. First we take a subset $Y'$ of exactly $dn$ elements of $Y$ and partition it into $n$ $d$-colorful $d$-tuples, by filling the elements of $Y'$ into an $n \times d$ grid, column by column, starting with the elements of $Y_1 \cap Y'$, then the elements of $Y_2 \cap Y'$, and so on. The rows of the grid then form the desired partition: since each $Y_i \cap Y$ has at most $n$ elements, no two elements in the same row get
the same color, and hence the $d$-tuple in each row is $d$-colorful.

Finally, we extend each $d$-tuple of the partition to a $k$-tuple by adding arbitrary $k-d$ elements of the remaining set $X\setminus Y'$. The resulting $k$-tuples are automatically $d$-colorful.
\end{proof}


\begin{figure}
\centering
\begin{tikzpicture}

\begin{scope}[scale=.6]
\draw (0,0) -- (5,0) -- (5,6) --(0,6) -- cycle;

\fill[blue!95!white, opacity = .85] (0,0)--(1,0)--(1,9)--(0,9)--cycle;
\fill[blue!70!white, opacity = .85] (1,0)--(2,0)--(2,8)--(1,8)--cycle;
\fill[blue!50!white, opacity = .85] (2,0)--(3,0)--(3,3)--(2,3)--cycle;
\fill[blue!30!white, opacity = .85] (3,0)--(4,0)--(4,2)--(3,2)--cycle;
\fill[blue!15!white, opacity = .85] (4,0)--(5,0)--(5,2)--(4,2)--cycle;
\draw[thin, blue!50!black] 
(0,0)--(5,0)
(0,1)--(5,1)
(0,2)--(5,2)
(0,3)--(3,3)
(0,4)--(2,4)
(0,5)--(2,5)
(0,6)--(2,6)
(0,7)--(2,7)
(0,8)--(2,8)
(0,9)--(1,9)
(0,0)--(0,9)
(1,0)--(1,9)
(2,0)--(2,8)
(3,0)--(3,3)
(4,0)--(4,2)
(5,0)--(5,2);
\node at (0.5,-.7){\footnotesize $X_1$};
\node at (1.5,-.7){\footnotesize $X_2$};
\node at (2.5,-.7){\footnotesize $X_3$};
\node at (3.5,-.7){\footnotesize $X_4$};
\node at (4.5,-.7){\footnotesize $X_5$};

\node at (-1.2,5.5) {\footnotesize $n=6$};
\node at (4.5,4.5) {\footnotesize $Y$};
\node at (.5,0.5) {\footnotesize $1$};
\node at (.5,1.5) {\footnotesize $1$};
\node at (.5,2.5) {\footnotesize $1$};
\node at (.5,3.5) {\footnotesize $1$};
\node at (.5,4.5) {\footnotesize $1$};
\node at (.5,5.5) {\footnotesize $1$};
\node at (.5,6.5) {\footnotesize $1$};
\node at (.5,7.5) {\footnotesize $1$};
\node at (.5,8.5) {\footnotesize $1$};
\node at (1.5,0.5) {\footnotesize $2$};
\node at (1.5,1.5) {\footnotesize $2$};
\node at (1.5,2.5) {\footnotesize $2$};
\node at (1.5,3.5) {\footnotesize $2$};
\node at (1.5,4.5) {\footnotesize $2$};
\node at (1.5,5.5) {\footnotesize $2$};
\node at (1.5,6.5) {\footnotesize $2$};
\node at (1.5,7.5) {\footnotesize $2$};
\node at (2.5,0.5) {\footnotesize $3$};
\node at (2.5,1.5) {\footnotesize $3$};
\node at (2.5,2.5) {\footnotesize $3$};
\node at (3.5,0.5) {\footnotesize $4$};
\node at (3.5,1.5) {\footnotesize $4$};
\node at (4.5,0.5) {\footnotesize $5$};
\node at (4.5,1.5) {\footnotesize $5$};
\end{scope}

\begin{scope}[scale=.6, xshift = 9.5cm]
\fill[blue!95!white, opacity = .85] (0,0)--(1,0)--(1,6)--(0,6)--cycle;
\fill[blue!70!white, opacity = .85] (1,0)--(2,0)--(2,6)--(1,6)--cycle;
\fill[blue!50!white, opacity = .85] (2,0)--(3,0)--(3,3)--(2,3)--cycle;
\fill[blue!30!white, opacity = .85] (2,3)--(3,3)--(3,5)--(2,5)--cycle;
\fill[blue!15!white, opacity = .85] (2,5)--(3,5)--(3,6)--(2,6)--cycle;
\draw[thin, blue!50!black] 
(0,0)--(3,0)
(0,1)--(3,1)
(0,2)--(3,2)
(0,3)--(3,3)
(0,4)--(3,4)
(0,5)--(3,5)
(0,6)--(3,6)
(0,0)--(0,6)
(1,0)--(1,6)
(2,0)--(2,6)
(3,0)--(3,6);

\node at (.5,0.5) {\footnotesize $1$};
\node at (.5,1.5) {\footnotesize $1$};
\node at (.5,2.5) {\footnotesize $1$};
\node at (.5,3.5) {\footnotesize $1$};
\node at (.5,4.5) {\footnotesize $1$};
\node at (.5,5.5) {\footnotesize $1$};
\node at (1.5,0.5) {\footnotesize $2$};
\node at (1.5,1.5) {\footnotesize $2$};
\node at (1.5,2.5) {\footnotesize $2$};
\node at (1.5,3.5) {\footnotesize $2$};
\node at (1.5,4.5) {\footnotesize $2$};
\node at (1.5,5.5) {\footnotesize $2$};
\node at (2.5,0.5) {\footnotesize $3$};
\node at (2.5,1.5) {\footnotesize $3$};
\node at (2.5,2.5) {\footnotesize $3$};
\node at (2.5,3.5) {\footnotesize $4$};
\node at (2.5,4.5) {\footnotesize $4$};
\node at (2.5,5.5) {\footnotesize $5$};
\node at (-.7,3) {\footnotesize $Y'$};

\end{scope}

\begin{scope}[scale=.6, xshift = 9.5cm, yshift=.2cm]

\fill[blue!95!white, opacity = .85] (0,6)--(1,6)--(1,9)--(0,9)--cycle;
\fill[blue!70!white, opacity = .85] (1,6)--(2,6)--(2,8)--(1,8)--cycle;
\fill[blue!15!white, opacity = .85] (2,6)--(3,6)--(3,7)--(2,7)--cycle;
\draw[thin, blue!50!black] 
(0,6)--(3,6) 
(0,7)--(3,7) 
(0,8)--(2,8)
(0,9)--(1,9)
(0,6)--(0,9)
(1,6)--(1,9)
(2,6)--(2,8)
(3,6)--(3,7);
\node at (0.5,6.5) {\footnotesize $1$};
\node at (0.5,7.5) {\footnotesize $1$};
\node at (0.5,8.5) {\footnotesize $1$};
\node at (1.5,6.5) {\footnotesize $2$};
\node at (1.5,7.5) {\footnotesize $2$};
\node at (2.5,6.5) {\footnotesize $5$};
\node at (-1.2,7.5) {\footnotesize $X\setminus  Y'$};

\end{scope}

\begin{scope}[scale=.6, xshift = 17cm]
\fill[blue!95!white, opacity = .85] (0,0)--(1,0)--(1,6)--(0,6)--cycle;
\fill[blue!70!white, opacity = .85] (1,0)--(2,0)--(2,6)--(1,6)--cycle;
\fill[blue!50!white, opacity = .85] (2,0)--(3,0)--(3,3)--(2,3)--cycle;
\fill[blue!30!white, opacity = .85] (2,3)--(3,3)--(3,5)--(2,5)--cycle;
\fill[blue!15!white, opacity = .85] (2,5)--(3,5)--(3,6)--(2,6)--cycle;
\draw[thin, blue!50!black] 
(0,0)--(3,0)
(0,1)--(3,1)
(0,2)--(3,2)
(0,3)--(3,3)
(0,4)--(3,4)
(0,5)--(3,5)
(0,6)--(3,6)
(0,0)--(0,6)
(1,0)--(1,6)
(2,0)--(2,6)
(3,0)--(3,6);
\node at (-.7,0.4) {\footnotesize $P_1$};
\node at (-.7,1.4) {\footnotesize $P_2$};
\node at (-.7,2.4) {\footnotesize $P_3$};
\node at (-.7,3.4) {\footnotesize $P_4$};
\node at (-.7,4.4) {\footnotesize $P_5$};
\node at (-.7,5.4) {\footnotesize $P_6$};

\node at (.5,0.5) {\footnotesize $1$};
\node at (.5,1.5) {\footnotesize $1$};
\node at (.5,2.5) {\footnotesize $1$};
\node at (.5,3.5) {\footnotesize $1$};
\node at (.5,4.5) {\footnotesize $1$};
\node at (.5,5.5) {\footnotesize $1$};
\node at (1.5,0.5) {\footnotesize $2$};
\node at (1.5,1.5) {\footnotesize $2$};
\node at (1.5,2.5) {\footnotesize $2$};
\node at (1.5,3.5) {\footnotesize $2$};
\node at (1.5,4.5) {\footnotesize $2$};
\node at (1.5,5.5) {\footnotesize $2$};
\node at (2.5,0.5) {\footnotesize $3$};
\node at (2.5,1.5) {\footnotesize $3$};
\node at (2.5,2.5) {\footnotesize $3$};
\node at (2.5,3.5) {\footnotesize $4$};
\node at (2.5,4.5) {\footnotesize $4$};
\node at (2.5,5.5) {\footnotesize $5$};
\end{scope}

\begin{scope}[scale=.6, xshift = 20.15cm]
\fill[blue!95!white, opacity = .85] (0,0)--(1,0)--(1,3)--(0,3)--cycle;
\fill[blue!70!white, opacity = .85] (0,3)--(1,3)--(1,5)--(0,5)--cycle;
\fill[blue!15!white, opacity = .85] (0,5)--(1,5)--(1,6)--(0,6)--cycle;
\draw[thin, blue!50!black] 
(0,0)--(1,0)
(0,1)--(1,1)
(0,2)--(1,2)
(0,3)--(1,3)
(0,4)--(1,4)
(0,5)--(1,5)
(0,6)--(1,6)
(0,0)--(0,6)
(1,0)--(1,6);

\node at (.5,0.5) {\footnotesize $1$};
\node at (.5,1.5) {\footnotesize $1$};
\node at (.5,2.5) {\footnotesize $1$};
\node at (.5,3.5) {\footnotesize $2$};
\node at (.5,4.5) {\footnotesize $2$};
\node at (.5,5.5) {\footnotesize $5$};

\end{scope}

\end{tikzpicture}
\caption{Partition into parts $P_1, \dots, P_6$ with parameters $k=4$, $n=6$, $m=5$, and $d=3$. \label{fig:skyscraper}}
\end{figure}
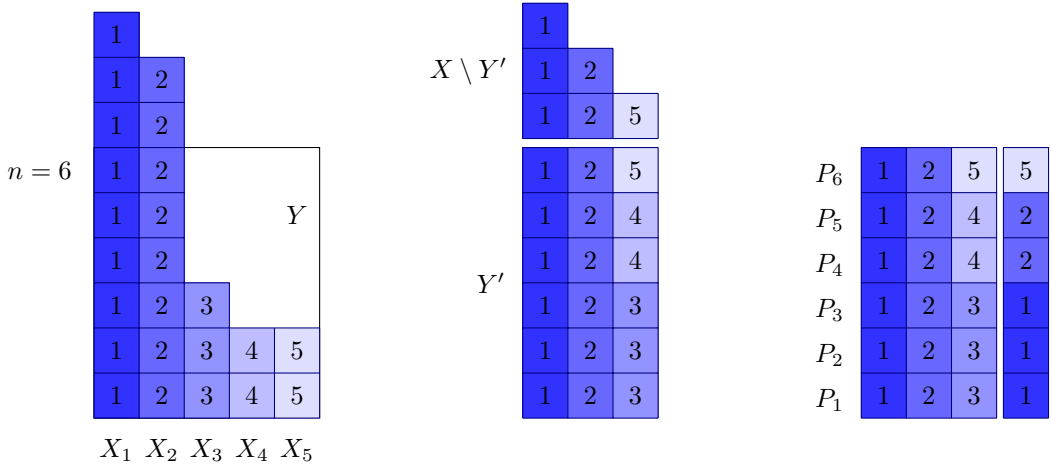

\subsubsection*{Merging colors}

In order to prove Conjecture~\ref{conj_main} for $k=d$, it is enough to prove it for $m=2d-1$~\cite{kano}. Indeed, if the number of color classes is larger, we can merge two classes of small size into a single class of size at most $n$, which implies that condition~\eqref{eq_hall} is still satisfied.

The next lemma implies that to prove Conjecture~\ref{conj_main} for $k>d$, it is enough to prove it for $m=2d-2$.

\begin{lemma}
\label{lemma_merge}
Let $d,k,m$ be integers satisfying $d \ge 2$, $k\ge d+1$, and $m\ge 2d-1$. Let $X$ be a set with $kn$ elements and let $X=X_1 \cup X_2 \cup \dots \cup X_m$ be an $m$-coloring of $X$ satisfying condition~\eqref{eq_hall}. There exist two color classes such that by merging them into a single color class, the resulting $(m-1)$-coloring of $X$ still satisfies condition~\eqref{eq_hall}.
\end{lemma}

\begin{proof}
Assume that $|X_1|\ge|X_2|\ge \dots\ge|X_m|$. We merge $X_{m-1}$ and $X_{m}$ into a single class $X'_{m-1}$. Since $m\ge 2d-1$, we have $|X'_{m-1}|=|X_{m-1}|+|X_m|\le 2kn/(2d-1)$. We now verify that the $(m-1)$-coloring $\chi'=(X_1, X_2, \dots, X_{m-2}, X'_{m-1})$ satisfies~(\ref{eq_hall}).

Suppose that $t\in [d-1]$ is the smallest integer for which $\chi'$ violates condition~\eqref{eq_hall}. The inequality $2kn/(2d-1) \le (k-d+1)n$ follows from $(2d-3)(k-d)\ge 1$, thus $t\ge 2$. Then $|X'_{m-1}|>|X_t|$, and so $|X'_{m-1}| + \sum_{i=1}^{t-1} |X_i| > (k-d+t)n$ while $\sum_{i=1}^{t-1} |X_i| \le (k-d+t-1)n$.
By our assumption, we have
$|X'_{m-1}|\le \frac{2}{m-t+1}\cdot\sum_{i=t}^{m} |X_i|$. Together, this gives
\begin{align*}
(m-t+1)(k-d+t)n &<
2\cdot\sum_{i=t}^{m} |X_i| + (m-t-1)\cdot \sum_{i=1}^{t-1} |X_i| \\
&\le 2 \cdot\sum_{i=1}^{m} |X_i| + (m-t-1)\cdot(k-d+t-1)n \\
&= 2kn + (m-t-1)(k-d+t-1)n
\end{align*}
and thus
\[
0>(m-t+1)(k-d+t)-2k - (m-t-1)(k-d+t-1) = m-t+2k-2d+2t-1-2k\ge t-2\ge 0;
\]
a contradiction.
\end{proof}

The following examples show that it is not always possible to merge two color classes if $m=2d-1$ and $k=d$ or if $m=2d-2$ and $k\ge d+1$.

For $m=2d-1$ and $k=d$, let $n$ be a multiple of $2d-1$, let $X$ be a set with $dn$ elements and let $X=X_1 \cup X_2 \cup \dots \cup X_{2d-1}$ be a $(2d-1)$-coloring of $X$ satisfying
\[
|X_1|=|X_2|=\dots = |X_{2d-1}| = \frac{d}{2d-1}\cdot n.
\]
Then by merging an arbitrary pair of color classes we get a class with $(2dn)/(2d-1) > n$ elements, violating condition~\eqref{eq_hall}.

For $m=2d-2$ and $k\ge d+1$, let $n$ be a multiple of $2d-3$, let $X$ be a set with $kn$ elements and let $X=X_1 \cup X_2 \cup \dots \cup X_{2d-2}$ be a $(2d-2)$-coloring of $X$ satisfying
\[
|X_1|=(k-d+1)n \quad \text{ and }\quad |X_2|=|X_3|=\dots =|X_{2d-2}|=\frac{d-1}{2d-3}\cdot n.
\]
Let $Y=X_{2d-3}\cup X_{2d-2}$. Now $X=X_1 \cup Y \cup X_2 \cup X_3\cup \dots \cup X_{2d-4}$ is a $(2d-3)$-coloring of $X$ where $X_1$ and $Y$ are the two largest color classes. We have
\[
|X_1|+|Y|=\left(k-d+1+\frac{2d-2}{2d-3}\right)\cdot n > (k-d+2)n,
\]
which violates condition~\eqref{eq_hall}.

\section{Proof of Theorem~\ref{theorem_plane}}
\label{section_planar}
Our proof of Theorem~\ref{theorem_plane} is a modification of the proof by Kaneko, Kano and Suzuki~\cite[Theorem 2]{kanekoKanoSuzuki} and relies on the following crucial result by Bespamyatnikh, Kirkpatrick and Snoeyink~\cite[Theorem 5]{besp}, restated by Kaneko et al.~\cite[Theorem 6]{kanekoKanoSuzuki} in a more general form, which we also use. More precisely, the formulation by Bespamyatnikh et al.~\cite{besp} assumes that $a_1/b_1 = a_2/b_2 = a_3/b_3$, but Kaneko et al.~\cite{kanekoKanoSuzuki} observed that the proof can be easily modified so that the assumption can be omitted. See Figure~\ref{fig:3fan}.

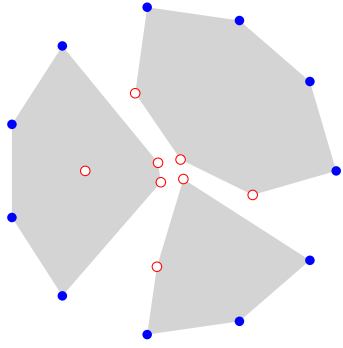
\begin{figure}
\centering
\begin{tikzpicture}[scale=0.8]
 \newcommand\Square[1]{+(-#1,-#1) rectangle +(#1,#1)}
\begin{scope}
\fill[gray, opacity=.34] (0*360/11: 2.75cm) -- (1*360/11: 2.75cm) -- (2*360/11: 2.75cm) -- (3*360/11: 2.75cm) -- (360*7/22  :1.42cm) -- (360*3/22  :.25cm) -- (360*21/22  :1.42cm);
\fill[gray, opacity=.34] (4*360/11: 2.75cm) -- (5*360/11: 2.75cm) -- (6*360/11: 2.75cm) -- (7*360/11: 2.75cm)  -- (360*14/22  :.25cm) -- (360*9/22  :.25cm) ;
\fill[gray, opacity=.34] (8*360/11: 2.75cm) -- (9*360/11: 2.75cm) -- (10*360/11: 2.75cm) -- (360*20/22  :.25cm)-- (360*16/22  :1.61cm) ;
\end{scope}
\begin{scope}[rotate = 0]
\draw[red,fill=white] (360*7/22  :1.42cm)  circle (2.25pt);
\draw[red,fill=white] (360*11/22  :1.42cm)  circle (2.25pt);
\draw[red,fill=white] (360*16/22  :1.61cm)  circle (2.25pt);
\draw[red,fill=white] (360*21/22  :1.42cm)  circle (2.25pt);
\draw[red,fill=white] (360*9/22  :.25cm)  circle (2.25pt);
\draw[red,fill=white] (360*14/22  :.25cm)  circle (2.25pt);
\draw[red,fill=white] (360*20/22  :.25cm)  circle (2.25pt);
\draw[red,fill=white] (360*3/22  :.25cm)  circle (2.25pt);
\end{scope}
\begin{scope}[rotate = 0]
\foreach \i in {1,...,11} { \fill[blue] (360*\i/11  :2.75cm) circle (2.25pt); }
\end{scope}
\end{tikzpicture}
\caption{The 3-cutting theorem with parameters $a_1 = 3$,  $a_2 = 4$, $a_3=4$ and $b_1 = 2$, $b_2 = 3$, $b_3 = 3$. \label{fig:3fan}}
\end{figure}

\begin{theorem}[$3$-cutting theorem~{\cite{besp,kanekoKanoSuzuki}}]
\label{theorem_3cuts} Let $a_1,a_2, a_3, b_1, b_2, b_3$ be positive integers.
Let $A$ and $B$ be finite disjoint sets of points in the plane such that $A\cup B$ is in general position, $|A| = a_1+a_2+a_3$, and $|B|=b_1+b_2+b_3$. Suppose that any open halfplane containing exactly $a_i$ points from $A$ contains less than $b_i$ points from $B$. Then there exist disjoint convex sets $C_1, C_2, C_3$ such that $|C_i\cap A| = a_i$ and $|C_i \cap B| = b_i$ for every $i\in\{1,2,3\}$.
\end{theorem}

Our proof of Theorem~\ref{theorem_plane} actually gives a slightly stronger conclusion. In particular, the $k$-islands form a ``near-equipartition" in the sense that the numbers of points of a given color in distinct $k$-islands differ by at most $1$.

\begin{proof}[Proof of Theorem~\ref{theorem_plane}]
We denote the two color classes of $X$ by $A$ and $B$, so $X=A\cup B$. We proceed by induction on $n$. The statement is trivial for $n=1$. If $|A|$ and $|B|$ are both divisible by $n$, then the result follows from Theorem~\ref{theorem_sakai}. We may therefore assume that there are positive integers $a, b, s, t$ such that
\[|A| = na+s, \;\;\; |B| = nb+t , \;\;\; k = a+b+1, \;\text{ and }\;\; s+t = n.\]
We claim that there exist pairwise disjoint $k$-islands $C_1,\allowbreak C_2, \dots,\allowbreak C_{s},\allowbreak D_1,\allowbreak D_2, \dots, \allowbreak D_{t}$ such that $|C_j\cap A| = a+1$, $|C_j\cap B| = b$, $|D_i\cap A|= a$, and $|D_i\cap B| = b+1$ for every $j\in[s], i\in[t]$.

For a fixed integer $i\in[t]$, consider an open halfplane $H_i$ containing precisely $ia$ points from $A$. If $|H_i\cap B| = i(b+1)$, then the complement of $H_i$ contains $(n-i)a +s$ points from $A$ and $(n-i)b +(t-i)$ points from $B$, and we are done by induction. We may therefore assume that $H_i$ contains strictly less than $i(b+1)$ points or strictly more than $i(b+1)$ points. The following observation is well-known (see for instance \cite[Lemma 3]{besp}) and can be shown by a simple continuity argument.

\begin{observation}
\label{obs_signwell}
Let $i\in[t]$ and let $H_i$ and $H_i'$ be two open halfplanes, each containing exactly $ia$ points from $A$. If $|H_i\cap B|<i(b+1)$ and $|H'_i\cap B|>i(b+1)$, then there exists a halfplane $H''_i$ satisfying $|H''_i\cap A|=ia$ and $|H''_i\cap B|=i(b+1)$.
\end{observation}

In view of Observation~\ref{obs_signwell} we may assume that either every open halfplane containing exactly $ia$ points from $A$ contains less than $i(b+1)$ points from $B$, or every open halfplane containing exactly $ia$ points from $A$ contains less than $i(b+1)$ points from $B$. We denote these two cases by $\sigma_a(i) = -1$ and $\sigma_a(i) = +1$, respectively.

By the same argument, for every fixed integer $j\in[s]$, either every open halfplane containing exactly $j(a+1)$ points from $A$ contains less than $jb$ points from $B$, or every open halfplane containing exactly $j(a+1)$ points from $A$ contains more than $jb$ points from $B$. We denote these two cases by $\sigma_{a+1}(i) = -1$ and $\sigma_{a+1}(i) = +1$, respectively.

Under the assumption that there is no open halfplane containing exactly $a$ points from $A$ and $(b+1)$ points from $B$, or $(a+1)$ points from $A$ and $b$ points from $B$, we observe the following.

\begin{observation}
$\sigma_a(1) = \sigma_{a+1}(1)$.
\end{observation}

\begin{proof}
To see why this holds, consider a line $l$ passing through one point from $A$ and no point from $B$, and has precisely $a$ points from $A$ on its left side. Let $l'$ be a line parallel to $l$ slightly to the right of $l$ such that no points from $A\cup B$ are contained in the open strip bounded by $l$ and $l'$. Thus the two open left halfplanes bounded by $l$ and $l'$ contain the same number of points from $B$, which is either smaller than $b$ or greater than $b+1$.
\end{proof}

Without loss of generality, we may assume that $\sigma_a(1) = \sigma_{a+1}(1) = -1$. (Otherwise, we can just exchange the roles of $A$ and $B$.) We now claim that if the sequence $\sigma_a(1), \sigma_a(2), \sigma_a(3), \dots$, changes signs, then we can find parameters satisfying the conditions of Theorem \ref{theorem_3cuts}. To see this, suppose there exists a smallest integer $i\le t$ such that \[\sigma_a(1) = \sigma_a(2) = \cdots = \sigma_a(i-1) = -1 \; \text{ and } \; \sigma_a(i) = +1.\] Consider a line $l$ disjoint with $A\cup B$ that has exactly $ia$ points from $A$ on its left side. By the assumption $\sigma_a(i)=+1$, it follows that on the right side of $l$ there are exactly $|A|-ia = (n-i)a + s$ points from $A$ and less than $|B|-i(b+1) = (n-i)b+(t-i)$ points from $B$. Therefore, the hypothesis of Theorem~\ref{theorem_3cuts} is satisfied with
\begin{align*}
a_1 &= a,  & a_2 &= (i-1)a,      & a_3 &= (n-i)a+s \\
b_1 &= b+1,&  b_2 &= (i-1)(b+1), & b_3 &= (n-i)b + (t-i).
\end{align*}

We can thus apply the inductive hypothesis in each of the resulting convex sets $C_1$, $C_2$, $C_3$ guaranteed by Theorem~\ref{theorem_3cuts}.

By the same argument applied to the sequence $\sigma_{a+1}(1), \sigma_{a+1}(2), \dots, \sigma_{a+1}(s)$, we may assume that $\sigma_a(i) = \sigma_{a+1}(j) = -1$ for all $i\in[t]$ and $j\in[s]$. In particular, $\sigma_{a+1}(s) = \sigma_a(t)$, but this is a contradiction since these signs correspond to complementary halfplanes.
\end{proof}

\section{Proof of Theorem~\ref{theorem_Rd}}
\label{section_simplices}

The proof of Theorem~\ref{theorem_Rd} is similar to the proof of Theorem~\ref{theorem_kyncl}~\cite{kano}, but it is a bit easier, since here we can use directly the continuous ham sandwich theorem, instead of its generalization.

\begin{theorem}[The ham sandwich theorem~\cite{ST42_generalized}, {\cite[Theorem 3.1.1]{Mato2003}}]
\label{theorem_ham_sandwich}
Let $\mu_1,\mu_2,\dots,\allowbreak \mu_d$ be $d$ absolutely continuous finite Borel measures on
$\mathbb{R}^d$. Then there exists a hyperplane $h$ such that
each open halfspace $H$ defined by $h$ satisfies
$\mu_i(H)=\mu_i(\mathbb{R}^d)/2$ for every $i\in [d]$.
\end{theorem}

Theorem~\ref{theorem_Rd} follows by induction from the following special discrete version of the ham sandwich theorem, which is an analogue of the discrete hamburger theorem from~\cite{kano}.

We say that point sets $X_1, X_2, \dots, X_{d}$ are \emph{balanced} in a subset $S\subseteq \mathbb{R}^d$ if for every $i\in[d]$, we have
\[
\lvert S\cap X_i\rvert \ge \frac{1}{d+1}\cdot \sum_{j=1}^{d} \lvert S\cap X_j\rvert.
\]

\begin{theorem}
\label{theorem_special_discrete_sandwich}
Let $d\ge 2$ and $n\ge 2$ be integers.
Let $X_1, X_2, \dots, X_{d}\subset \mathbb{R}^d$ be $d$ disjoint point sets balanced in $\mathbb{R}^d$. Suppose that $X_1\cup X_2 \cup\cdots \cup X_d$ is in general position and $\sum_{i=1}^{d+1} |X_i|=(d+1)n$. Then there exists a hyperplane $h$ disjoint with each $X_i$ such that for each open halfspace $H$ determined by $h$, the sets $X_1, X_2,\dots, X_{d}$ are balanced in $H$ and $\sum_{i=1}^{d} |H\cap X_i|$ is a positive integer multiple of $d+1$.
\end{theorem}

\subsection{Proof of Theorem~\ref{theorem_special_discrete_sandwich}}

Let $X=\bigcup_{i=1}^{d} X_i$.
Replace each point $\mathbf{p}\in X$ by an open ball $B(\mathbf{p})$ of a sufficiently small radius $\delta>0$ centered at $\mathbf{p}$, so that no hyperplane intersects or touches more than $d$ of these balls. We will apply the ham sandwich theorem for suitably defined measures supported by the balls $B(\mathbf{p})$. Rather than taking the same measure for each of the balls, we use a variation of the trick used by Elton and Hill~\cite{EH11_stronger_conclusion}. For each $\mathbf{p}\in X$ and $k\ge 1$, we choose a small number $\varepsilon_k(\mathbf{p})\in (0, 1/k)$ so that for every $i\in [d]$
and for every subset $Y_i\subset X_i$, we have
\begin{equation}\label{eq_epsilony}
\sum_{\mathbf{p}\in Y_i} (1-\varepsilon_k(\mathbf{p})) \neq \frac{1}{2} \cdot \sum_{\mathbf{p}\in X_i} (1-\varepsilon_k(\mathbf{p})).
\end{equation}
Now let $k\ge 1$ be a fixed integer. For each $i\in [d]$, let $\mu_{i,k}$ be the measure supported by the closure of $\bigcup_{\mathbf{p}\in X_i} B(\mathbf{p})$ such that it is uniform (that is, equal to a multiple of the Lebesgue measure) on each of the balls $B(\mathbf{p})$ and $\mu_{i,k}(B(\mathbf{p}))=1-\varepsilon_k(\mathbf{p})$.

We apply the ham sandwich theorem (Theorem~\ref{theorem_ham_sandwich}) to the measures $\mu_{i,k}$, $i\in [d]$, and obtain a bisecting hyperplane $h_k$.

For each $i\in [d]$, let $\mu_i$ be the limit of the measures $\mu_{i,k}$ as $k$ tends to infinity; that is, $\mu_i$ is uniform on every ball $B(\mathbf{p})$ such that $\mathbf{p}\in X_i$ and $\mu_i(B(\mathbf{p}))=1$. Since the supports of all the measures $\mu_{i,k}$ are uniformly bounded, there is a sequence $\{k_m\}$ such that the subsequence of hyperplanes $h_{k_m}$ has a limit $h'$. More precisely, if $h_{k_m}=\{\mathbf{x}\in \mathbb{R}^d; \mathbf{x} \cdot \mathbf{v}_m=c_m\}$ where $\mathbf{v}_m \in S^{d-1}$, then $h'=\{\mathbf{x}\in \mathbb{R}^d; \mathbf{x} \cdot \mathbf{v}=c\}$ where $\mathbf{v} = \lim_{m\rightarrow \infty} \mathbf{v}_m$ and $c = \lim_{m\rightarrow \infty} c_m$. By the absolute continuity of the measures, each open halfspace $H$ defined by $h'$ satisfies $\mu_i(H)=\mu_i(\mathbb{R}^d)/2$ for every $i\in [d]$.

The condition~\eqref{eq_epsilony} ensures that for every $m$, the hyperplane $h_{k_m}$ intersects the support of each measure $\mu_{i,k_m}$, $i\in [d]$. In particular, for each $i\in [d]$, there is a point $\mathbf{p}_i\in X_i$ such that for infinitely many $m\ge 1$, the hyperplane $h_{k_m}$ intersects each of the balls $B(\mathbf{p}_i)$. It follows that for each $i\in [d]$, the hyperplane $h'$ either touches $B(\mathbf{p}_i)$ or intersects $B(\mathbf{p}_i)$. In fact, $h'$ touches $B(\mathbf{p}_i)$ if $|X_i|$ is even and $h'$ contains the point $\mathbf{p}_i$ if $|X_i|$ is odd, and each open halfplane determined by $h'$ contains exactly $\lfloor|X_i|/2\rfloor$ points of $X_i$. 

We now rotate the hyperplane $h'$ slightly, to a hyperplane $h$ that touches each of the balls $B(\mathbf{p}_i)$, so that the point sets $X_1, X_2, \dots, X_d$ are balanced in each halfspace determined by $h$ and the number of points of $X$ in each halfspace is divisible by $d+1$. Essentially, for each point $\mathbf{p}_i$ we can decide independently on which side of $h$ it will end.
We consider two cases according to the parity of $n$.


\paragraph{Case 1.} Assume that $n=2n'$ for some positive integer $n'$. Since the point sets $X_1,\allowbreak X_2, \dots,\allowbreak X_d$ are balanced in $\mathbb{R}^2$, we have $|X_i|\ge 2n'$ for each $i\in[d]$. To satisfy the balancing condition for the two halfspaces, each halfspace must contain at least $n'$ points from each $X_i$. This is already satisfied for the original hyperplane $h'$ and each $X_i$ with an even number of points. If $|X_i|$ is odd for some $i$, then $|X_i|\ge 2n'+1$ and thus moving the point $\mathbf{p}_i$ to either side of $h$ will keep at least $n'$ points of $X_i$ in each halfspace. Since $|X|$ is even, there is an even number of $X_i$'s with odd cardinality, and therefore $|h'\cap X|$ is also even. Since $|X|/2=n'(d+1)$, the divisibility condition will be satisfied if each halfspace gets exactly $|X|/2$ points of $X$. This is easily achieved if we move half of the points from $h'\cap X$ to one halfspace and the remaining points of $h'\cap X$ to the other halfspace. 

\paragraph{Case 2.} Assume that $n=2n'+1$ for some positive integer $n'$. Then $|X|=(2n'+1)(d+1)$, and so the only way of satisfying the divisibility condition is having $(n'+1)(d+1)$ points in one halfspace 
and $n'(d+1)$ points in the other halfspace determined by $h$. Since we can move $d$ points between the halfspaces, this determines the number of points $\mathbf{p}_i$ that have to end in each halfspace determined by $h$. 

Since $h'$ bisects each of the measures $\mu_i$ and $\mu_i(\mathbb{R}^2)=|X_i|$ for each $i$, we have to move $(d+1)/2$ units of the total measure $\mu=\sum_{i=1}^d \mu_i$ from one halfspace to the other one. Moving a point $\mathbf{p}_i$ contained in $h'$ to one halfspace corresponds to moving half unit of $\mu$, and moving a point $\mathbf{p}_i$ from one open halfspace to the opposite open halfspace corresponds to moving one unit of $\mu$.

Assume without loss of generality that $h'$ is horizontal, so that $h$ will be almost horizontal. We will refer to the two halfspaces determined by $h'$ or $h$ as the halfspace \emph{above} and \emph{below} $h'$ or $h$, respectively.

Let $c=|h'\cap X|$; that is, $c$ is the number of sets $X_i$ with odd cardinality.
Let $a$ be the number of the points $\mathbf{p}_i$ above $h'$, and let $b$ be the number of the points $\mathbf{p}_i$ below $h'$. Clearly, $a+b+c=d$. Assume without loss of generality that $a\ge b$. Since $c$ has the same parity as $|X|=(2n'+1)(d+1)$, then $d+1-c$ is even, and thus $a+b=d-c$ is odd. This implies that $a\ge b+1$, and consequently $c/2+a\ge (d+1)/2$.

We move all the $c$ points of $h'\cap X$ below $h$, and an arbitrary set of $(d+1-c)/2$ points $\mathbf{p}_i$ that are above $h'$ to the halfspace below $h$. After that, the halfspace below $h$ has exactly $(n'+1)(d+1)$ points of $X$. See Figure~\ref{fig:rounding}.

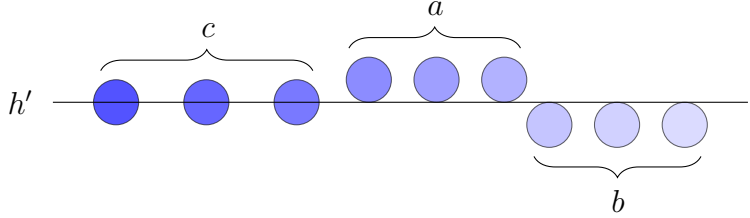
\begin{figure}
\centering
\begin{tikzpicture}[scale = .6]
\fill[blue!90!white, opacity =.75]
(-2.6,0) circle (.5cm);
\fill[blue!80!white, opacity =.75]
(-.6,0) circle (.5cm);
\fill[blue!70!white, opacity =.75]
(1.4,0) circle (.5cm);
\fill[blue!60!white, opacity =.75]
(3,0.5) circle (.5cm);
\fill[blue!50!white, opacity =.75]
(4.5,.5) circle (.5cm);
\fill[blue!40!white, opacity =.75]
(6,0.5) circle (.5cm);
\fill[blue!30!white, opacity =.75]
(7,-0.5) circle (.5cm);
\fill[blue!24!white, opacity =.75]
(8.5,-0.5) circle (.5cm);
\fill[blue!19!white, opacity =.75]
(10,-0.5) circle (.5cm);

\draw[opacity = .6] 
(-2.6,0) circle (.5cm)
(-.6,0) circle (.5cm)
(1.4,0) circle (.5cm)
(3,0.5) circle (.5cm)
(4.5,.5) circle (.5cm)
(6,0.5) circle (.5cm)
(7,-0.5) circle (.5cm)
(8.5,-0.5) circle (.5cm)
(10,-0.5) circle (.5cm);

\draw (-4,0) -- (11.4,0);

\draw [decorate,decoration={brace,amplitude=8pt},xshift=-4pt,yshift=0pt] (-2.8,.7) -- (1.9,.7) ;

\draw [decorate,decoration={brace,amplitude=8pt},xshift=-4pt,yshift=0pt] (2.7,1.2) -- (6.5,1.2) ;

\draw [decorate,decoration={brace,amplitude=8pt},xshift=-4pt,yshift=0pt]   (10.5,-1.2) -- (6.8,-1.2);

\node[left] at (-4.2,0) {$h'$};
\node[above] at (-.56,1.2) {$c$};
\node[above] at (4.48,1.7) {$a$};
\node[below] at (8.52,-1.7) {$b$};

\end{tikzpicture}
\caption{A ham sandwich cut by a hyperplane $h'$ for $m=d=9$, $k=10$ and $n$ odd. The hyperplane $h'$ intersects or touches $d= a + b + c$ of the balls, one from each $X_i$. We perturb $h'$ so that the $c$ balls with centers in $h'$ and $(d+1-c)/2$ of the balls tangent to $h'$ from above end up below the perturbed hyperplane.}
\label{fig:rounding}
\end{figure}

We now verify that the balancing condition is satisfied in both halfspaces.
Since $|X_i|\ge 2n'+1$ for every $i$, every set $X_i$ of odd cardinality has at least $n'$ points in the halfspace 
above $h$ and at least $n'+1$ points below $h$. Every set $X_i$ of cardinality at least $2n'+3$ has at least $n'+1$ points in each halfspace determined by $h$. Every set $X_i$ of cardinality $2n'+2$ with $\mathbf{p}_i$ below $h'$ has exactly $n'+1$ points in each halfspace determined by $h$. Finally, every set $X_i$ of cardinality $2n'+2$ with $\mathbf{p}_i$ above $h'$ has $n'$ or $n'+1$ points above $h$, and $n'+2$ or $n'+1$ points below $h$.

\section{Final remarks}

Conjecture~\ref{conj_main} is still open for $d\ge 3$ and $k,m\ge d+1$. It is likely that generalizing the $3$-cutting theorem~{\cite{besp,kanekoKanoSuzuki}} to $\mathbb{R}^d$, up to $d+1$ parts and $2d-2$ color classes might be a fruitful approach to prove Conjecture~\ref{conj_main} in full generality.

Several proofs of special cases of Conjecture~\ref{conj_main} include a step where a partition theorem for measures is discretized into a corresponding partition theorem for point sets; see Sakai's proof of the $3$-cutting theorem~\cite{sakai}, our Theorem~\ref{theorem_special_discrete_sandwich} or the discrete version of the hamburger theorem~\cite{kano}. However, there are difficulties with this approach already for $d=3$ and $k=m=4$: Figure~\ref{fig_discretization_counterexample2} shows that cutting by a single hyperplane is not sufficient.

\begin{figure}
\centering
\begin{tikzpicture}[scale = .6]
\fill[blue!90!white, opacity =.75]
(-4,1) circle (.5cm)
(-2.5,1) circle (.5cm)
(-1,1) circle (.5cm);
\fill[blue!90!white, opacity =.75]
(-4,-1) circle (.5cm)
(-2.5,-1) circle (.5cm)
(-1,-1) circle (.5cm);

\fill[blue!70!white, opacity =.75]
(2,.5) circle (.5cm)
(2,-1) circle (.5cm);

\fill[blue!50!white, opacity =.75]
(5,.5) circle (.5cm)
(5,-1) circle (.5cm);

\fill[blue!30!white, opacity =.75]
(8,1) circle (.5cm)
(8,-.5) circle (.5cm);

\draw[opacity = .6] 
(-4,1) circle (.5cm)
(-2.5,1) circle (.5cm)
(-1,1) circle (.5cm)
(-4,-1) circle (.5cm)
(-2.5,-1) circle (.5cm)
(-1,-1) circle (.5cm)
(2,.5) circle (.5cm)
(2,-1) circle (.5cm)
(5,.5) circle (.5cm)
(5,-1) circle (.5cm)
(8,1) circle (.5cm)
(8,-.5) circle (.5cm);

\draw (-6,0) -- (10,0);

\end{tikzpicture}
\caption{An example for $d=3$ and $k=m=4$ showing that the discretization approach from the proof of Theorem~\ref{theorem_special_discrete_sandwich} does not generalize easily. The figure represents a ham sandwich cut for four measures in $\mathbb{R}^3$. The cutting hyperplane, represented by the horizontal line, touches the supports of three different measures, but it cannot be locally modified to give a discrete balanced partition; that is, a partition satisfying the divisibility condition and the analogue of condition~\eqref{eq_hall} simultaneously.}
\label{fig_discretization_counterexample2}
\end{figure}
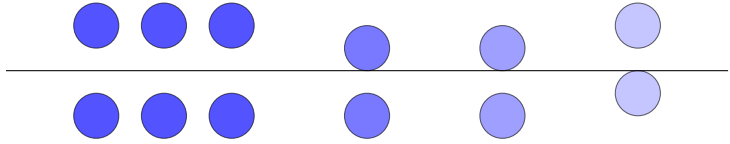

We were not able to prove Conjecture~\ref{conj_main} even in the case when $X$ has the order type of the vertex set of the cyclic polytope, when one might expect the existence of a purely combinatorial solution.

\section{Acknowledgements}
We are grateful to the reviewers for simplifying the rounding step in the proof of Theorem~\ref{theorem_special_discrete_sandwich}, and for their suggestions for improving the presentation.


\end{document}